\documentclass[a4paper,12pt,reqno,intlimits]{extarticle}
\usepackage[cp1251]{inputenc}
\usepackage[T2A]{fontenc}
\usepackage[english]{babel}
\usepackage{amsfonts,amsmath,amsxtra,amsthm,amssymb,latexsym,mathrsfs,graphicx,mathtools,cite,nicefrac,url}
\usepackage{algorithm}
\usepackage{algorithmic}
\usepackage{multirow}
\usepackage{authblk}

\theoremstyle{plain}
\newtheorem{theorem}{Theorem}
\newtheorem{lemma}{Lemma}

\newtheorem{corollary}{Corollary}

\theoremstyle{definition}
\newtheorem{remark}{Remark}
\newtheorem{example}{Example}

\numberwithin{theorem}{section}
\numberwithin{lemma}{section}
\numberwithin{proposition}{section}
\numberwithin{corollary}{section}
\numberwithin{remark}{section}
\numberwithin{definition}{section}
\numberwithin{example}{section}

\DeclarePairedDelimiter\bracket{(}{)}
\newcommand{\br}[1]{\bracket*{#1}}
\DeclarePairedDelimiter\figbracket{\{}{\}}
\newcommand{\fbr}[1]{\figbracket*{#1}}

\DeclarePairedDelimiter\modbracket{|}{|}
\newcommand{\mbr}[1]{\modbracket*{#1}}
\DeclarePairedDelimiter\angbracket{\langle}{\rangle}
\newcommand{\abr}[1]{\angbracket*{#1}}
\DeclarePairedDelimiter\normbracket{\|}{\|}
\newcommand{\nbr}[1]{\normbracket*{#1}}

\title{TITLE}
\date{}

\title{Mirror Descent for Constrained Optimization Problems with Large Subgradient Values\footnote{Stonyakin F.S., Stepanov A.N., Gasnikov A.V., Titov A.A. Mirror descent for constrained optimization problems with large subgradient values of functional constraints // Computer Research and Modeling, 2020, vol. 12, no. 2, pp. 301-317.}}

\author[1,2]{Fedor~S.~Stonyakin}
\author[1]{Alexey~N.~Stepanov}
\author[2]{Alexander~A.~Titov}
\author[2]{Alexander~V.~Gasnikov}

\affil[1]{V.\,I.\,Vernadsky Crimean Federal University, Simferopol, Russia,

\text{\small email: fedyor@mail.ru, stepanov.student@gmail.com}.}

\affil[2]{Moscow Institute of Physics and Technology, Moscow, Russia,

\text{\small email: a.a.titov@phystech.edu, gasnikov@yandex.ru}.}

\begin{document}

\maketitle

\begin{abstract}
Based on the ideas of \cite{bib_Adaptive}, we consider the problem of minimization of the Lipschitz-continuous non-smooth functional $f$ with a non-positive convex (generally, non-smooth) Lipschitz-continuous functional constraint. We propose some novel strategies of step-sizes and adaptive stopping rules in Mirror Descent algorithms for the considered class of problems. It is shown that the methods are applicable to the objective functionals of various levels of smoothness. Applying the restart technique to the Mirror Descent Algorithm there was proposed an optimal method of solving optimization problems with strongly convex objective functionals. Estimates of the rate of convergence of the considered algorithms are obtained depending on the level of smoothness of the objective functional. These estimates indicate the optimality of considered methods from the point of view of the theory of lower oracle bounds. In addition, the case of a quasi-convex objective and functional constraint was considered.

\end{abstract}


\section{Introduction}

Non-smooth convex constrained optimization problems play an important role in modern large-scale optimization and its applications \cite {bib_ttd,bib_Shpirko}. There are a lot of methods to solve such problems, among which one can mention the Mirror Descent Method  \cite{beck2003mirror,nemirovsky1983problem}.


Recently, in \cite{bib_Adaptive} algorithms for Mirror Descent with both adaptive step selection and adaptive stopping criterion were proposed. In addition, an optimal method was proposed for the special class of convex constrained optimization problems, when the gradient of the objective functional satisfies Lipschitz property. For example, quadratic functionals do not satisfy the Lipschitz condition, but their gradient does. An adaptive Mirror Descent algorithm, based on the ideology of \cite{bib_Nesterov,bib_Nesterov2016} was proposed to solve such problems in (\cite{bib_Adaptive}, Section 3.3).

In this paper we develop the above mentioned research and consider some modifications of algorithmic scheme (\cite{bib_Adaptive}, Section 3.3). More precisely, in proposed Algorithm 2 we consider a new approach to choosing a step in the method, as well as appropriate options for stopping criteria, which differ from \cite{bib_Adaptive}. It is important that we choose the non-productive step ($\nabla g(x^k)$ is the subgradient $g$ at the current point $x^k$) at the form $h_k = \frac{\varepsilon}{\|\nabla g(x^k)\|}$ instead of $h_k = \frac{\varepsilon}{\|\nabla g(x^k)\|^2}$ in \cite{bib_Adaptive}. This circumstance, as well as the appropriate choice of the number of iterations \eqref{stop_rule}, leads us to the fact that the method can run faster than the previous analogue (\cite{bib_Adaptive}, Section 3.3) in the case, when the values of the subgradients of the  functional constraint $g$ are large. Note that a method similar to Algorithm 2 was proposed in \cite{nemirovsky1983problem}  for the case of convex Lipschitz continuous functionals.

This paper substantiates the convergence rate estimates for the proposed version of the Mirror Descent method, proves its optimality from the point of view of the theory of lower bounds for objective functionals of various smoothness levels: which have a Lipschitz continuous gradient or satisfy the Lipschitz (H\"{o}lder) condition. It is also shown that the obtained estimates of the convergence rate are preserved for quasi-convex \cite{Nest_84,Konnov} objective functional and constraint (see e.g. \cite{Gas}, Exercise 2.7). 
Using the restart technique, the optimal method for strongly (quasi-)convex objective functionals is considered. The paper ends with some numerical experiments for geometric problems with functional constraints, which illustrate , that the proposed method can work faster compared to \cite{bib_Adaptive}, Section 3.3. There are also given some examples of more efficient methods in the case of large dimensionality.

The contribution of this paper is as follows:

- An analogue of the Mirror Descent method is considered (\cite{bib_Adaptive}, Section 3.3) for convex programming problems with another strategy for choosing a non-productive step. Estimates of the rate of its convergence and optimality are obtained in terms of lower bounds for convex objective functionals of various smoothness levels.

- It is shown that the obtained convergence rate estimates will are also valid for the case of the minimization problems with quasi-convex objective functionals of different smoothness levels.

- It is shown that for the H\"{o}lder-continuous quasi-convex objective the convergence rate is equal to $O\br{\frac{1}{\varepsilon^2}}$.

- Using the restart technique, an optimal method was proposed for the class of minimization problems with strongly (quasi-)convex H\"{o}lder-continuous objective functionals with the complexity estimate equal to $O\left(\frac{1}{\varepsilon}\right)$.

- Numerical experiments for geometrical problems (the Fermat-Torricelli-Steiner problem, the problem of the smallest covering ball) with convex constraints are presented. When (sub)gradient values of functional constraints are large the proposed method can work faster \cite{bib_Adaptive}. High-dimensional examples are also considered.

- Numerical experiments for the minimization of quasi-convex functionals are given.
An example of the smallest covering ball problem with a quasi-convex objective functional is also considered.





\section{Problem Statement and Mirror Descent Basics}

Let $(E,||\cdot||)$ be a normed finite-dimensional vector space and $E^*$ be its conjugate space with the norm:
$$||y||_*=\max\limits_x\{\langle y,x\rangle,||x||\leqslant1\},$$
where $\langle y,x\rangle$ is the value of the continuous linear functional $y$ at $x \in E$.

Let $Q\subset E$ be a (simple) closed convex set. Consider the following problem:
\begin{equation}\label{eqq2}
f(x) \rightarrow \min\limits_{x\in Q},
\end{equation}
s.t.
\begin{equation}
\label{problem_statement_g}
g(x) \leqslant 0
\end{equation}
Assume that convex functional $g$ satisfies the Lipschitz condition with a constant $M_g$:
\begin{equation}\label{eq1}
|g(x)-g(y)|\leqslant M_g\|x-y\|\;\forall x,y\in X.
\end{equation}

We consider cases of convex and quasi-convex objective $f$. Let $d : Q \rightarrow \mathbb{R}$ be a distance generating function (d.g.f) which is non-negative continuously differentiable and $1$-strongly convex w.r.t. the norm $\|\cdot\|$, i.e.
$$\forall x, y, \in Q \hspace{0.2cm} \langle \nabla d(x) - \nabla d(y), x-y \rangle \geqslant \| x-y \|^2,$$
and assume that there is a constant $\Theta_0$, such that $d(x_{*}) \leqslant \Theta_0^2,$ where $x_*$ is a solution of the problem (supposing that the problem is solvable).

For all $x, y\in Q \subset E$ consider the corresponding Bregman divergence
$$V(x, y) = d(y) - d(x) - \langle \nabla d(x), y-x \rangle.$$
The proximal mapping operator is defined as follows:
$$
\mathrm{Mirr}_x (p) = \arg\min\limits_{u\in Q} \big\{ \langle p, u \rangle + V(x, u) \big\} \;  \text{ for each } \;  x\in Q \; \text{ and } \;  p\in E^*.
$$
We assume for simplicity that $\mathrm{Mirr}_x (p)$ is easily computable.

\section{Mirror Descent Algorithms: New Step-Sizes Strategies}

Two Mirror Descent methods for optimization problems with one convex subdifferentiable functional constraint  were proposed in \cite{bib_Adaptive}. The convergence of the first of them is obtained for the case of the Lipschitz-continuous objective functional (see \cite{bib_Adaptive}, Section 3.1), while the convergence of the second is justified under the assumption that the gradient $\nabla f$ satisfies Lipschitz property (see \cite {bib_Adaptive}, p. 3.3). Let us remind namely, the second one.

\begin{algorithm}
\caption{Adaptive Mirror Descent}
\label{alg2}
\begin{algorithmic}[1]
\REQUIRE $\varepsilon>0,\Theta_0: \,d(x_*)\leqslant\Theta_0^2$
\STATE $x^0=argmin_{x\in X}\,d(x)$
\STATE $I=:\emptyset$
\STATE $N\leftarrow0$
\REPEAT
    \IF{$g(x^N)\leqslant\varepsilon$}
        \STATE $M_N=\|\nabla f(x^N)\|_*$, $h_N=\frac{\varepsilon}{M_N}$
        \STATE $x^{N+1}=Mirr_{x^N}(h_N\nabla f(x^N))\text{ // \emph{"productive steps"}}$
        \STATE $N\rightarrow I$
    \ELSE
        \STATE $M_N=\|\nabla g(x^N)\|_*$, $h_N=\frac{\varepsilon}{M_N^2}$
        \STATE $x^{N+1}=Mirr_{x^N}(h_N\nabla g(x^N))\text{ // \emph{"non-productive steps"}}$
    \ENDIF
    \STATE $N\leftarrow N+1$
\UNTIL{$2\frac{\Theta_0^2}{\varepsilon^2}\leqslant\sum_{j\not\in I}\frac{1}{M_j^2}+|I|$}
\ENSURE $\bar{x}^N:=argmin_{x^k,\;k\in I}\,f(x^k)$
\end{algorithmic}
\end{algorithm}


\begin {lemma}\label{LemmT}
Let us define the following function:
\begin {equation} \label {eq13}
\omega (\tau) = \max \limits_ {x \in X} \{f (x) -f (x _ *): \| x-x _ * \| \leqslant \tau \},
\end {equation} where $\tau$ is a positive number.
Then for any $ y \in X$
\begin {equation} \label {eq_lemma}
f (y) - f (x_*) \leqslant \omega (v_f (y, x_ *)),
\end {equation}
where
\begin{equation}\label{eqv_v_f}
v_f(y,x_*) = \left \langle \frac{\nabla f(y)}{\|\nabla f(y)\|}, y - x_*\right\rangle \text{ for } \nabla f(y) \neq 0
\end{equation}
and $v_f(y,x_*) = 0$ for $\nabla f(y) = 0$.
\end {lemma}

For Algorithm \ref{alg2} the following theorem is valid.

\begin{theorem}\label{th2}
Let $\varepsilon> 0 $ be a fixed number and the stopping criterion of Algorithm \ref{alg2} is satisfied. Then
\begin {equation} \label {eq09}
\min \limits_ {k \in I} v_f (x ^ k, x _ *) <\varepsilon, \ \ \max\limits_{k\in I} g(x^k)\leqslant \varepsilon
\end {equation}
Note, that Algorithm \ref {alg2} works no more than
\begin {equation} \label {eqq08}
N = \left \lceil \frac {2 \max \{1, M_g ^ 2 \} \Theta_0 ^ 2} {\varepsilon ^ 2} \right \rceil
\end {equation}
iterations.
\end {theorem}

Now we will estimate the rate of convergence of the proposed method. For this we need the following auxiliary assumption (\cite {bib_Nesterov}, Lemma 3.2.1). Recall that $ x_* $ ~ is the solution of the problem \eqref {eqq2} - \eqref {problem_statement_g}.

Basing on lemma \ref{LemmT} and theorem \ref {th2}, we can estimate the rate of convergence of Algorithm \ref {alg2} for a differentiable objective functional $ f $ with the Lipschitz-continuous gradient:
\begin {equation} \label {eqlipgrad}
\|\nabla f (x) - \nabla f (y)\|_* \leqslant L \|x-y\| \quad \forall x, y \in X.
\end {equation}

Using the well-known inequality for exact solution $x_*$ (see, for example, \cite {bib_Nesterov})
$$ f (x) \leqslant f (x_*) + \| \nabla f (x_*) \|_* \| x-x_* \| + \frac {1}{2} L \| x-x_* \|^2$$
we can get that
$$ \min\limits_ {k \in I} f (x^k) -f (x_*) \leqslant \min \limits_ {k \in I} \left \{\| \nabla f (x_*) \|_* \| x^k-x_* \| + \frac{1}{2} L \| x^k-x_* \| ^ 2 \right \}. $$
Further, the following estimate is valid:
$$ f (x) -f (x_*) \leqslant \varepsilon  \| \nabla f (x_*) \|_* + \frac{1}{2} L \varepsilon^2. $$

\begin {corollary} \label {cor1}
Let $f$ be differentiable on $X$ and \eqref {eqlipgrad} holds. Then, after the stopping of Algorithm \ref {alg2}, the next inequality holds:
$$ \min \limits_ {1 \leqslant k \leqslant N} f (x ^ k) -f (x _ *) \leqslant \varepsilon_f + \frac {L \varepsilon ^ 2} {2} =   \varepsilon \cdot \| \nabla f (x _ *) \| _ *  + \frac {L \varepsilon ^ 2} {2}.  $$

\end {corollary}

\if
Let us consider a special class of non-smooth objective functionals.

\begin{corollary}\label{cor2}
Let $f(x) = \max\limits_{i = \overline{1, m}} f_i(x)$, where $f_i$ is differentiable $\forall x \in X$ and
$$\|\nabla f_i(x)-\nabla f_i(y)\|_*\leqslant L_i\|x-y\| \quad \forall x,y\in X.$$
Then after Algorithm \ref{alg4} stops, the next inequality holds: $$\min\limits_{0\leqslant k\leqslant N}f(x^k)-f(x_*)\leqslant\varepsilon_f+\frac{L\varepsilon^2}{2},$$
where
$$\varepsilon_f=\varepsilon \cdot \|\nabla f(x_*)\|_*, \quad L = \max\limits_{i = \overline{1, m}} L_i.$$
\end{corollary} \fi

Let us observe a new version of the adaptive Mirror Descent method with another step selection strategy. A resembling idea was researched in \cite{Nem-Yud} for Lipschitz-continuous functional. Note, that the following modification can be used to minimize functionals with different levels of smoothness. As earlier, we will consider the method for a fixed accuracy $\varepsilon>0$, an initial approximation $x^0$, and some value $\Theta_0$, such that $V(x^0,x_*)\leqslant \Theta_0^2$.

\begin{algorithm}
\caption{Adaptive Mirror Descent}
\label{alg4}
\begin{algorithmic}[1]
\REQUIRE $\varepsilon>0,\Theta_0: \,d(x_*)\leqslant\Theta_0^2$
\STATE $x^0=argmin_{x\in X}\,d(x)$
\STATE $I=:\emptyset$
\STATE $N\leftarrow0$
\REPEAT
    \IF{$g(x^N)\leqslant\varepsilon\|\nabla g(x^N)\|_*$}
        \STATE $M_N=\|\nabla f(x^N)\|_*$, $h_N=\frac{\varepsilon}{M_N}$
        \STATE $x^{N+1}=Mirr_{x^N}(h_N\nabla f(x^N))\text{ // \emph{"productive steps"}}$
        \STATE $N\rightarrow I$
    \ELSE
        \STATE $M_N=\|\nabla g(x^N)\|_*$, $h_N=\frac{\varepsilon}{M_N}$
        \STATE $x^{N+1}=Mirr_{x^N}(h_N\nabla g(x^N))\text{ // \emph{"non-productive steps"}}$
    \ENDIF
    \STATE $N\leftarrow N+1$
\UNTIL{$2\frac{\Theta_0^2}{\varepsilon^2}\leqslant N$}
\ENSURE $\bar{x}^N:=argmin_{x^k,\;k\in I}\,f(x^k)$
\end{algorithmic}
\end{algorithm}


Note, that Algorithm  \ref{alg4} works during a fixed number of steps
\begin{equation}\label{stop_rule}
N = \left\lceil\frac{2\Theta_0^2}{\varepsilon^2}\right\rceil.
\end{equation}
The following theorem holds.

\begin{theorem}\label{th2_new_methods}
Let $\varepsilon > 0$ be a fixed number and the stopping criterion of Algorithm \ref{alg4} be satisfied. Then
\begin{equation}\label{eq1_new_methods}
\min\limits_{k \in I} v_f(x^k,x_*) \leqslant \varepsilon, \ \ \max\limits_{k\in I} g(x^k)\leqslant\varepsilon M_g.
\end{equation}
\end{theorem}

\begin{proof}

1) If $k\in I$,
\begin{equation}\label{eq2}
\begin{split}
h_k\abr{\nabla f(x^k),x^k-x_*} = \varepsilon v_f(x^k,x_*)\leqslant\\
\leqslant\frac{h_k^2}{2}||\nabla f(x^k)||_*^2+V(x^k,x_*)-V(x^{k+1},x_*)=\\
=\frac{\varepsilon^2}{2}+V(x^k,x_*)-V(x^{k+1},x_*).
\end{split}
\end{equation}

2) If $k\not\in I$, then $\frac{g(x^k)}{||\nabla g(x^k)||_*}>\varepsilon$ and $\frac{g(x^k)-g(x_*)}{||\nabla g(x^k)||_*}\geqslant\frac{g(x^k)}{||\nabla g(x^k)||_*}>\varepsilon$. Therefore, the following inequalities hold
\begin{equation}\label{eq3}
\begin{split}
\varepsilon^2<h_k\br{g(x^k)-g(x_*)}\leqslant\frac{h_k^2}{2}||\nabla g(x^k)||_*^2+\\
+V(x^k,x_*)-V(x^{k+1},x_*)=\frac{\varepsilon^2}{2}+V(x^k,x_*)-V(x^{k+1},x_*),\text{ or}\\
\frac{\varepsilon^2}{2}<V(x^k,x_*)-V(x^{k+1},x_*).
\end{split}
\end{equation}

3) After summing up the inequalities \eqref{eq2} and \eqref{eq3} one can get
$$\sum_{k\in I} \varepsilon v_f(x^k,x_*) \leqslant |I|\frac{\varepsilon^2}{2}-\frac{\varepsilon^2|J|}{2}+V(x^0,x_*)-V(x^{k+1},x_*)=$$
$$ =\varepsilon^2|I|-\frac{\varepsilon^2 N}{2}+\Theta_0^2.$$

After the stopping criterion of the algorithm holds one can get $$\min\limits_{k \in I} v_f(x^k,x_*) \leqslant \varepsilon.$$

Further, for each  $k\in I\;\;\;g(x^k)\leqslant\varepsilon||\nabla g(x^k)||_*\leqslant\varepsilon M_g$ and
$$ g(\hat{x})\leqslant\frac{1}{\sum_{k\in I}h_k}\sum_{k\in I}h_kg(x^k)\leqslant\varepsilon M_g.$$

Now we have to show that the set of productive steps $I$ is non-empty. If $I=\emptyset$, then $|J|=N$ and \eqref{eq1} means, that $N\geqslant\frac{2\Theta_0^2}{\varepsilon^2}$. On the other hand, from \eqref{eq3} we have:
$$\frac{\varepsilon^2N}{2}<V(x^0,x_*)\leqslant\Theta_0^2,$$
which leads us to the controversy, so $I\neq\emptyset$.
\end{proof}

Let us show how to estimate the quality of the solution by the function basing on the previous theorem. Note, that it is possible to take into account different levels of smoothness of the objective functional.

\begin{corollary}\label{cor_lipschits1}
Let $f$ satisfy the Lipschitz condition
\begin{equation}\label{lipschits_condition}
|f(x)-f(y)|\leqslant M_f\|x-y\|\quad \forall x,y\in X.
\end{equation}
Then, after the stopping of Algorithm \ref{alg4}, the following inequality holds:
$$\min\limits_{k\in I}f(x^k)-f(x_*)\leqslant M_f \varepsilon.$$
\end{corollary}

\section{The case of a quasi-convex functionals}
Let us consider the optimization problem \eqref{eqq2} under the assumption of quasi-convexity of the objective functional $f$. The case of both quasi-convex $f$ and functional constraint  $g$ is observed in the Conclusions.
Recall (see \cite{Haz}) that function $\phi:Q\rightarrow \mathbb{R}$ is called quasi-convex, if
$$\phi\left((1-\alpha)x+\alpha y\right)\leqslant \max\{\phi(x),\phi(y)\} \ \ \forall \alpha \in [0;1]\ \ \forall x,y\in Q.$$
As earlier, let $g$ satisfy Lipschitz condition \eqref{eq1} with the constant $M_g$.

Let us remind the definition of Clarke subdifferential. Let $x_0 \in \mathbb{R}^n$ be a fixed point and $h \in \mathbb{R}^n$ be a fixed direction. Denote
$$f_{Cl}^{\uparrow}(x_0;h)=\lim\limits_{x'\rightarrow x_0}\sup\limits_{\alpha \downarrow 0}\frac{1}{\alpha}\br{ f(x'+\alpha h) - f(x') }.$$

Value $f_{Cl}^{\uparrow}(x_0;h)$ is called Clarke subdifferential of functional $f$ at the point $x_0$ in the direction $h$. This function is subadditive and positively homogeneous, thus we can define the subdifferential of the  function $f$ at the point $x_0$ as follows:
$$\partial_{Cl}f(x_0):=\left\{ v\in \mathbb{R}\ |\ f_{Cl}^{\uparrow}(x_0;g)\geqslant vg \ \forall g\in\mathbb{R} \right\}.$$
According to this, $$f_{Cl}^{\uparrow}(x_0;h) = \max\limits_{v\in \partial_{Cl}f(x_0)}\langle v,h \rangle. $$

Note, that from now we will understand any element (vector) of the Clarke subdifferential as the subgradient of the quasi-convex (locally Lipschitz) functional $f$. For convex functional $g$, we understand the concept of a subgradient in the standard way.

\begin{lemma}
Let $f:X\rightarrow \mathbb{R}.$ For any $y\in Q$, vector $p_y\in E^*$ and $h>0$ define $z=Mirr_y(h\cdot p_y)$. Then for any $x\in Q$ the next inequality holds:
$$h\langle p_y, y-x\rangle \leqslant \frac{h^2}{2}\|p_y\|^2_*+V(y,x)-V(z,x).$$
\end{lemma}
Note, that for convex subdifferentiable functional $f$ and subgradient $p_y=\nabla f(y)$ this inequality is modified as follows:
$$h(f(y)-f(x))\leqslant \langle\nabla f(y),y-x\rangle \leqslant \frac{h^2}{2}\|\nabla f(y)\|^2_*+V(y,x)-V(z,x).$$

Note that for quasi-convex objective $f$ and constraint $g$ instead of (sub)gradient $\nabla f(y)$ in $v_f(y, x_*)$ (see \eqref{eqv_v_f}) we can consider a normal vector $\nabla f(y)$ to a set of level of $f$ at point $y$ \cite{bib_Nesterov}. However, the (sub)gradient or Clarke subdifferential of $f$ or $g$ also can be used, if they are finite and nonzero.

\begin{theorem}
Let $f$ be a quasi-convex functional. Then for Algorithm 1 after \eqref{stop_rule} steps \eqref{eq1_new_methods} holds.
\end{theorem}

\begin{remark}\label{cor_gelder1}
Let $f$ satisfy the H\"{o}lder condition ($\nu\in[0;1)$)
\begin{equation}\label{gelder_condition}
|f(x)-f(y)|\leqslant M_{f,\;\nu}\|x-y\|^{\nu}\quad \forall x,y\in X.
\end{equation}

For example, $f(x)=\sqrt{x}$ and $f(x)=\sqrt[4]{x}$.

Let us note ($M_{f,\;\nu} \leqslant M_{\nu}$) the following inequality (\cite{Gas}, section~5; see \cite{Stonyakin_Gasnikov_2019} too)
\begin{equation}\label{eq002}
M_{\nu}a^{\nu}\leqslant M_{\nu}\br{\frac{M_{\nu}}{\delta}}^{\frac{1-\nu}{1+\nu}}\frac{a^2}{2}+\delta,
\end{equation}
which is true for small enough $\delta>0$ for $a \geqslant \varepsilon_0 > 0$: $f$ is a Lipshitz-continuous at $\varepsilon_0$-neighborhood of $x_*$ from \eqref{eq1_new_methods} for fixed $\varepsilon_0 > 0$ due to (sub)differentiability of $f$. Then by \eqref{gelder_condition} we have
$$\mbr{f(x)-f(y)}\leqslant\frac{M_{\nu}^{\frac{2}{1+\nu}}}{2\delta^{\frac{1-\nu}{1+\nu}}}\nbr{x-y}^2+\delta.$$

Set $\delta=\varepsilon$, $\varepsilon > \varepsilon_0$. Then
\begin{equation}\label{eq003}
\mbr{f(x)-f(y)}\leqslant\underbrace{\frac{M_{\nu}^{\frac{2}{1+\nu}}}{2\varepsilon^{\frac{1-\nu}{1+\nu}}}}_{M}\nbr{x-y}^2+\varepsilon.
\end{equation}

Then by Lemma \ref{LemmT} after the stopping of Algorithm 2, $\varepsilon_0 \leqslant \min\limits_{k\in I}v_f(x^k,x_*)<\varepsilon$ means the following inequality holds
\begin{equation}\label{eq004}
f(\widehat{x})-f^*\leqslant\frac{M_{\nu}^{\frac{2}{1+\nu}}}{2\varepsilon^{\frac{1-\nu}{1+\nu}}}\varepsilon^2+\varepsilon =
\frac{M_{\nu}^{\frac{2}{1+\nu}}}{2} \varepsilon^{1+\frac{2\nu}{1+\nu}} + \varepsilon.
\end{equation}
Note that for $\varepsilon<1$ the inequality \eqref{eq004} means
$$f(\widehat{x})-f^* \leqslant \widehat{M}\varepsilon$$
for some $\hat{M} > 0$. The another case of $\min\limits_{k\in I}v_f(x^k,x_*)< \varepsilon \leqslant \varepsilon_0$ is not interesting due to Lipshitz-continuity of $f$ at $\varepsilon_0$-neighborhood of $x_*$. So, for problems with (quasi)convex H\"{o}lder-continuous (sub)differentiable objective and convex Lipshitz-continuous functional constraints we can achieve an $\varepsilon$-solution after
$$
O\left(\frac{1}{\varepsilon^2}\right)
$$
iterations of the Mirror Descent method. Obviously, this estimate is optimal. This estimate is optimal due to its optimality on a significantly narrower class of problems with Lipschitz-continuous objective functionals
\end{remark}

\section{Optimal methods for Mirror Descent on a class of non-smooth strongly convex problems}

Consider the optimization problem under the assumption of strong convexity of the objective function and functional constraint with the parameter $\mu$.
\begin{equation}\label{acc_eq14}
f(x)\rightarrow\min,\;\;g(x)\leqslant 0,\;\;x\in X
\end{equation}
where $X$ is a closed convex set.

Let the prox function $d(x)$ be bounded on the unit sphere with respect to the chosen norm $ \| \cdot \|$:
\begin{equation}
\label{acc_eq:dUpBound}
d(x) \leqslant \Omega^{2}, \quad \forall x\in X : \|x \| \leqslant 1.
\end{equation}
Let $x^0 \in X$ and there exists $R_0 >0$, such that $\| x^0 - x_* \|^2 \leqslant R_0^2$.

We will propose methods which can guarantee an \textit{$\varepsilon$-solution} $\hat{x}$ of the problem (\ref{acc_eq14}):
$$f(\hat{x}) - f(x_*) \leqslant \varepsilon \text{ and } g(\hat{x}) \leqslant \varepsilon.$$

The main idea is using the restart technique of Algorithm \ref{alg4}. Consider one well-known  statement (see \cite{bayandina2018primal-dual}).

\begin{lemma}\label{acc_lem3}
Let $f$ and $g$ be $\mu$-strongly convex functionals with respect to the norm $\|\cdot\|$ on $X$, $x_{\ast} = \arg\min\limits_{x \in X} f(x)$, $g(x)\leqslant 0$ ($\forall x \in X$) and for some $\varepsilon_{f}>0$ and $\varepsilon_{g}>0$ the next inequalities hold:
\begin{equation}\label{acc_eq15}
f(x)-f(x_{\ast})\leqslant \varepsilon_{f},\;\;g(x)\leqslant\varepsilon_{g}.
\end{equation}
Then
\begin{equation}\label{acc_eq16}
\frac{\mu}{2}\|x-x_{\ast}\|^{2}\leqslant\max\{\varepsilon_{f},\varepsilon_{g}\}.
\end{equation}
\end{lemma}


Let us consider an analogue of Algorithm \ref{alg4} for strongly convex problems. We must emphasize that for Algorithm \ref{alg4} one can obtain effective estimates of the rate of convergence for the objective functionals with any level of smoothness. According to Remark \ref{cor_gelder1} we can apply our approach to H\"{o}lder-continuous objective functionals.

Let us consider, in particular, the following example.

Let $f(x) = \max\limits_{i = \overline{1, m}} f_i(x)$, where $f_i$ are differentiable at any $x \in X$ and their gradients are Lipschitz-continuous:
\begin{equation}\label{acc_condition_Lip}
\| \nabla f_i(x)-\nabla f_i(y) \|_*\leqslant L_i\| x-y \| \quad \forall x,y\in X,\;\forall i = \{1,m\}.
\end{equation}

Consider function $\tau: \mathbb{R}^{+}\rightarrow\mathbb{R}^{+}$:
\begin{equation}
\tau(\delta)=\max\left\{\delta\|\nabla f(x_{\ast})\|_{\ast}+\frac{\delta^{2}L}{2} , \; \delta \right\},
\end{equation}
where
$$L:= \max\limits_{i = \overline{1, m}}\{L_i\}.$$

It is obvious that $\tau $ decreases, $\tau(0) = 0$, so for any $\varepsilon>0$ there exists
$$\hat{\varphi}(\varepsilon)>0:\;\;\tau(\hat{\varphi}(\varepsilon))=\varepsilon.$$

\begin{algorithm}[]
	\caption{Restart procedure for Algorithm 3}
	\label{acc_algorithm2}
	\begin{algorithmic}[1]
		\REQUIRE $ \text{accuracy} \ \varepsilon>0; \text{initial point} \ x^0;$ $ \Omega \ \text{s.t.} \ d(x)\leqslant\Omega^2 \quad \forall x\in X: \|x \| \leqslant 1;$ $X; d(\cdot);$ $ \text{strong convexity parameter} \ \mu; R_0 \ \text{ such that} \ \| x^0 - x_* \|^2 \leqslant R_0^2.$
		\STATE Set $d_0(x) = d\left(\frac{x-x^0}{R_0}\right)$.
		\STATE Set $p=1.$
		\REPEAT
		\STATE Set $R_p^2 = R_0^2 \cdot 2^{-p}.$
		\STATE Set $\varepsilon_p = \frac{\mu R_p^2}{2}.$
		\STATE Set $x^p$ as output of Algorithm \ref{alg4} with accuracy $\hat{\varphi}(\varepsilon_p)$, prox function $d_{p-1}(\cdot)$ and $\Omega^2.$
		\STATE $d_p(x) \gets d\left(\frac{x - x^p}{R_p}\right)$.
		\STATE Set $p = p + 1.$
		\UNTIL $p>\log_2 \frac{\mu R_0^2}{2\varepsilon}.$
	\end{algorithmic}
\end{algorithm}

\begin{theorem}\label{acc_theorem2}
Let $\nabla f$ be Lipschitz-continuous, $f$ and $g$ be $\mu$-strongly convex on $X\subset\mathbb{R}^{n}$ and $d(x)\leqslant \Omega^2$ for all $x\in X,$ such that $\|x\|\leqslant1$. Let initial point $x^{0}\in X$ and $R_{0}>0$ satisfy
$$\|x^{0}-x_{\ast}\|^{2}\leqslant R^{2}_{0}.$$
Then for $\displaystyle{\widehat{p}=\left\lceil\log_{2}\frac{\mu R_{0}^{2}}{2\varepsilon}\right\rceil}$ output $x^{\widehat{p}}$ is an $\varepsilon$-solution of the problem (\ref{acc_eq14}), also, the following inequalities hold:
$$f(x^{\widehat{p}})-f(x_{*})\leqslant\varepsilon, \quad g(x^{\widehat{p}})\leqslant M_g\varepsilon,$$
$$\|x^{\widehat{p}}-x_{*}\|^{2}\leqslant\frac{2\varepsilon}{\mu}\max\{1,M_g\}.$$
The number of iterations of Algorithm \ref{alg4} during the work of Algorithm \ref{acc_algorithm2} will not exceeds
$$\widehat{p}+\sum_{p=1}^{\widehat{p}}\frac{2\Omega^2\max\{1,M_g\}}{\hat{\varphi}^{2}(\varepsilon_{p})},\;\;\text{where}\;\;\varepsilon_{p}=\frac{\mu R^{2}_{0}}{2^{p+1}}.$$
\end{theorem}
\text{The proof is given in Appendix.}

We can formulate the following corollary for the case $M_g \leqslant 1$.
\begin{corollary}
Let $\nabla f$ be Lipschitz-continuous, $f$ and $g$ be $\mu$-strongly convex on $X\subset\mathbb{R}^{n}$ and $d(x)\leqslant \Omega^2$ for all $x\in X,$ such that $\|x\|\leqslant1$. Let initial point $x^{0}\in X$ and $R_{0}>0$ satisfy
$$\|x^{0}-x_{\ast}\|^{2}\leqslant R^{2}_{0}.$$
Then for $M_g \leqslant 1$ and $\displaystyle{\widehat{p}=\left\lceil\log_{2}\frac{\mu R_{0}^{2}}{2\varepsilon}\right\rceil}$ output $x^{\widehat{p}}$ is an $\varepsilon$-solution of the problem (\ref{acc_eq14}), also, the following inequalities hold:
$$f(x^{\widehat{p}})-f(x_{*})\leqslant\varepsilon, \quad g(x^{\widehat{p}})\leqslant M_g\varepsilon,$$
$$\|x^{\widehat{p}}-x_{*}\|^{2}\leqslant\frac{2\varepsilon}{\mu}.$$
The number of iterations of Algorithm \ref{alg4} during the work of Algorithm \ref{acc_algorithm2} will not exceeds
$$\widehat{p}+\sum_{p=1}^{\widehat{p}}\frac{2\Omega^2}{\hat{\varphi}^{2}(\varepsilon_{p})},\;\;\text{where}\;\;\varepsilon_{p}=\frac{\mu R^{2}_{0}}{2^{p+1}}.$$
\end{corollary}

\begin{remark}
The estimate of the number of iterations of Algorithm \ref{alg4} can be detailed in the case of $\varepsilon < 1$. For any $\delta < 1$ there is such constant $C$, that $\tau(\delta) \leqslant C\delta$ for some constant $C$. So, we can suppose that $\hat{\varphi}(\varepsilon) = \widehat{C} \cdot \varepsilon$ for the corresponding constant $\widehat{C} > 0$.
On the restart number $p+1$ of Algorithm \ref{alg4} after no more than
\begin{equation}\label{acc_eq:kpp1Est}
k_{p+1} = \left\lceil\frac{2\Omega^2R_{p}^2}{\varepsilon_{p+1}^2}\right\rceil
\end{equation}
iterations of Algorithm \ref{alg4}, the output $x^{p+1}$  satisfies the following inequality:
$$f(x^{p+1})-f(x_*) \leqslant \widehat{C} \cdot \varepsilon_{p+1}, \quad g(x^{p+1}) \leqslant \varepsilon_{p+1},$$
where $\varepsilon_{p+1} = \frac{\mu R_{p+1}^2}{2}$.
According to Lemma \ref{acc_lem3},
$$\| x^{p+1} - x_* \|^2 \leqslant \frac{2 \max\{1, \widehat{C}\}\varepsilon_{p+1}}{\mu} = \max\{1, \widehat{C}\} \cdot R_{p+1}^2.$$
So, for all $p\geqslant 0$,
$$\|x^p-x_*\|^2\leqslant \max\{1, \widehat{C}\} \cdot R_p^2 = \max\{1, \widehat{C}\} \cdot R_0^2 \cdot 2^{-p}.$$

Note, that for all $p\geqslant 1$ the following inequalities hold:
$$f(x^{p})-f(x_*) \leqslant \max\{1, \widehat{C}\} \cdot \frac{\mu R_{0}^2}{2} \cdot 2^{-p}, \quad g(x_{p}) \leqslant \max\{1, \widehat{C}\} \cdot \frac{\mu R_{0}^2}{2} \cdot 2^{-p}.$$
Thereby, if $p > \log_2 \frac{\mu R_0^2}{2 \varepsilon}$, then $x^p$ will be $\left(\max\{1, \widehat{C}\}\varepsilon\right)$-solution to the problem, moreover:
$$\|x^p-x_*\|^2\leqslant \max\{1, \widehat{C}\} \cdot R_0^2 \cdot 2^{-p} \leqslant \frac{2 \varepsilon}{\mu}.$$

Let us evaluate the total number of iterations $N$ of Algorithm \ref{alg4}. Let $\hat{p} = \left\lceil \log_2 \frac{\mu R_0^2}{2\varepsilon}\right\rceil$. According to (\ref{acc_eq:kpp1Est}), up to multiplication by a constant we have:
\begin{align}
N &= \sum_{p=1}^{\hat{p}} k_p \leqslant \sum_{p=1}^{\hat{p}} \left(1 + \frac{2\Omega^2 R_{p}^2}{\varepsilon_{p+1}^2}\right) =
\sum_{p=1}^{\hat{p}} \left(1 + \frac{32\Omega^2 2^p}{\mu^2 R_0^2}\right) \notag \\
& \leqslant \hat{p} + \frac{64 \Omega^2 2^{ \hat{p} }}{\mu^2 R_0^2} \leqslant \hat{p} + \frac{64 \Omega^2}{\mu \varepsilon}. \notag
\end{align}
\end{remark}

Note, that the method can be applied to solve the problem \eqref{eqq2} in the case of strongly quasi-convex objective functional.
As earlier, $x_*$ is a solution of the optimization problem.
\begin{remark}

Function $f:Q\rightarrow \mathbb{R}$ is called strongly quasi-convex \cite{Nec}, if if for each $ x\in Q$
$$f(x_*)-f(x) \geqslant \langle \nabla f(x), x_*-x \rangle + \frac{\mu}{2}\|x_*-x\|^2,$$
where $x_*$ is a nearest to $x$ solution of the optimization problem.

Thus, the method and all the estimates in this  paragraph are valid in the case of strongly quasi-convex objective H\"{o}lder-continuous functionals.
\end{remark}

\section{Numerical Experiments} \label{section5}

All calculations were performed in CPython 3.7 on computer fitted with a 3-core AMD Athlon II X3 450 processor with a clock frequency of 3.2 GHz. The computer's RAM was 8 GB. We indicate the operating time of the algorithms in minutes and seconds.

\subsection{An analogue of the Fermat---Torricelli---Steiner problem}

\begin{example}\label{ex1}
Input data: $n=1000$, point coordinates\\
$A_k= (a_{1k}, a_{2k}, \ldots, a_{nk})$ ($k=1,2,\ldots,r$; $r=5$) are represented by integers from the interval $[-10,10]$, objective functional ($M_f=1$)
$$f(x)=\frac{1}{r}\sum\limits_{k=1}^{r}\sqrt{(x_1-a_{1k})^2+(x_2-a_{2k})^2+\ldots+(x_n-a_{nk})^2},$$
$x^0=\frac{(0.1,\ldots,0.1)}{\|(0.1,\ldots,0.1)\|}$, functional constraint
\begin{equation}\label{eq_gm}
\begin{split}
g(x)=\max\limits_{m=1,2,3,\ldots,20}\fbr{g_m(x)}\leqslant 0,\\
g_1(x)=\alpha_{11}|x_1|+\alpha_{12}|x_2|+\ldots+\alpha_{1n}|x_n|-1,\\
g_2(x)=\alpha_{21}|x_1|+\alpha_{22}|x_2|+\ldots+\alpha_{2n}|x_n|-1,\\
\ldots\\
g_m(x)=\alpha_{m1}|x_1|+\alpha_{m2}|x_2|+\ldots+\alpha_{mn}|x_n|-1,
\end{split}
\end{equation}
where the coefficients $\alpha_{11},\alpha_{12},\ldots,\alpha_{mn}$ are represented by the matrix
\begin{equation}\label{eq_matrix}
\begin{pmatrix}
1 & 1 & 1 & 1 & \dots & 1 & 1 \\
1 & 2 & 2 & 2 & \dots & 2 & 2 \\
1 & 3 & 3 & 3 & \dots & 3 & 3 \\
1 & 2 & 3 & 4 & \dots & 999 & 1000 \\
1 & 3 & 4 & 5 & \dots & 1000 & 1001 \\
\hdotsfor{7} \\
1 & 18 & 19 & 20 & \dots & 1015 & 1016 \\
\end{pmatrix}.
\end{equation}
\end{example}

The results of Example \ref{ex1} are presented in Table \ref{tab_ex1}. As one can observe, Algorithm \ref{alg4} works faster than Algorithm \ref{alg2}.

\begin{table}[]
\centering
\caption{Comparison of the results of the algorithms, Example \ref{ex1}.}
\label{tab_ex1}
\begin{tabular}{|c|c|c|c|c|}
\hline
\multirow{2}{*}{$\varepsilon$} & Iterations & \begin{tabular}[c]{@{}c@{}}Time,\\ MM:SS\end{tabular} & Iterations & \begin{tabular}[c]{@{}c@{}}Time,\\ MM:SS\end{tabular} \\ \cline{2-5}
 & \multicolumn{2}{c|}{Algorithm \ref{alg2}} & \multicolumn{2}{c|}{Algorithm \ref{alg4}} \\ \hline
\nicefrac{1}{2} & 30824 & 02:58 & 17 & 00:00.1 \\ \hline
\nicefrac{1}{4} & 61679 & 05:54 & 65 & 00:00.4 \\ \hline
\nicefrac{1}{6} & --- & >05:00 & 145 & 00:01 \\ \hline
\nicefrac{1}{8} & --- & >05:00 & 257 & 00:01 \\ \hline
\end{tabular}
\end{table}

\subsection{An analogue of the problem of the smallest covering circle}

\begin{example}\label{ex2}
Input data: $n=1000$, point coordinates\\
$A_k= (a_{1k}, a_{2k}, \ldots, a_{nk})$ ($k=1,2,\ldots,5$) are represented by integers from the interval $[-10,10]$, objective functional ($M_f=1$)
$$f(x)=\max\br{\sqrt{(x_1-a_{1k})^2+(x_2-a_{2k})^2+\ldots+(x_n-a_{nk})^2}},$$
$x^0=\frac{(0.1,\ldots,0.1)}{\|(0.1,\ldots,0.1)\|}$, functional constraint \eqref{eq_gm}, where the coefficients\\
$\alpha_{11},\alpha_{12},\ldots,\alpha_{mn}$ are represented by the matrix \eqref{eq_matrix}.
\end{example}

The results of Example \ref{ex2} are presented in Table \ref{tab_ex2}. As one can observe, Algorithm \ref{alg4} works faster than Algorithm \ref{alg2}.

\begin{table}[]
\centering
\caption{Comparison of the results of the algorithms, Example \ref{ex2}.}
\label{tab_ex2}
\begin{tabular}{|c|c|c|c|c|}
\hline
\multirow{2}{*}{$\varepsilon$} & Iterations & \begin{tabular}[c]{@{}c@{}}Time,\\ MM:SS\end{tabular} & Iterations & \begin{tabular}[c]{@{}c@{}}Time,\\ MM:SS\end{tabular} \\ \cline{2-5}
 & \multicolumn{2}{c|}{Algorithm \ref{alg2}} & \multicolumn{2}{c|}{Algorithm \ref{alg4}} \\ \hline
\nicefrac{1}{2} & 31264 & 03:01 & 17 & 00:00.1 \\ \hline
\nicefrac{1}{4} & 65056 & 06:16 & 65 & 00:00.4 \\ \hline
\nicefrac{1}{6} & --- & >05:00 & 145 & 00:01 \\ \hline
\nicefrac{1}{8} & --- & >05:00 & 257 & 00:01 \\ \hline
\end{tabular}
\end{table}

\subsection{An example of a concave objective functional satisfying the H\"{o}lder condition}

\begin{example}\label{ex3}
Input data: $n=1000$, objective functional ($M_{f,\nicefrac{1}{2}}=1$)
$$f(x)=\frac{1}{n}\sum\limits_{i=1}^{n}\sqrt{x_i},$$
$x^0=\frac{(0.1,\ldots,0.1)}{\|(0.1,\ldots,0.1)\|}$, $X=\{x=(x_1,\ldots,x_n)\;\mid\;x_i\geqslant0\;\;\forall i,\;\sum\limits_{i=1}^n x_i^2\leqslant1\}$,
functional constraint
\begin{equation}\label{eq_gm2}
\begin{split}
g(x)=\max\limits_{m=1,2,3,\ldots,20}\fbr{g_m(x)},\\
g_1(x)=\alpha_{11}x_1+\alpha_{12}x_2+\ldots+\alpha_{1n}x_n-1 \leqslant 0,\\
g_2(x)=\alpha_{21}x_1+\alpha_{22}x_2+\ldots+\alpha_{2n}x_n-1 \leqslant 0,\\
\ldots\\
g_m(x)=\alpha_{m1}x_1+\alpha_{m2}x_2+\ldots+\alpha_{mn}x_n-1 \leqslant 0,
\end{split}
\end{equation}
where the coefficients $\alpha_{11},\alpha_{12},\ldots,\alpha_{mn}$ are represented by the matrix \eqref{eq_matrix}.
\end{example}

The results of Example \ref{ex3} are presented in Table \ref{tab_ex3}. As one can observe, Algorithm \ref{alg4} works faster than Algorithm \ref{alg2}.


\begin{table}[]
\centering
\caption{Comparison of the results of the algorithms, Example \ref{ex3}.}
\label{tab_ex3}
\begin{tabular}{|c|c|c|c|c|}
\hline
\multirow{2}{*}{$\varepsilon$} & Iterations & \begin{tabular}[c]{@{}c@{}}Time,\\ MM:SS\end{tabular} & Iterations & \begin{tabular}[c]{@{}c@{}}Time,\\ MM:SS\end{tabular} \\ \cline{2-5}
 & \multicolumn{2}{c|}{Algorithm \ref{alg2}} & \multicolumn{2}{c|}{Algorithm \ref{alg4}} \\ \hline
\nicefrac{1}{2} & --- & >05:00 & 17 & 00:00.1 \\ \hline
\nicefrac{1}{4} & --- & >05:00 & 65 & 00:00.4 \\ \hline
\nicefrac{1}{6} & --- & >05:00 & 145 & 00:01 \\ \hline
\nicefrac{1}{8} & --- & >05:00 & 257 & 00:01 \\ \hline
\end{tabular}
\end{table}

\subsection{Examples with large dimensions}

Table \ref{tab_n3e5} presents the results of Algorithm \ref{alg4} for the dimension $n=3\cdot10^5$. Because of the large dimensionality it is impossible to obtain the results for Algorithm \ref{alg2} and its modified version, since the compiler composing the program code of the algorithm cannot process the input data due to the integer overflow error. Execution of Algorithm \ref{alg4} does not entail such an error.

\begin{table}[]
\centering
\caption{Some results of Algorithm \ref{alg4} for $n=3\cdot10^5$.}
\label{tab_n3e5}
\begin{tabular}{|c|c|c|c|c|c|c|}
\hline
\multirow{2}{*}{$\varepsilon$} & Iterations & \begin{tabular}[c]{@{}c@{}}Time,\\ MM:SS\end{tabular} & Iterations & \begin{tabular}[c]{@{}c@{}}Time,\\ MM:SS\end{tabular} & Iterations & \begin{tabular}[c]{@{}c@{}}Time,\\ MM:SS\end{tabular} \\ \cline{2-7}
 & \multicolumn{2}{c|}{Example \ref{ex1}} & \multicolumn{2}{c|}{Example \ref{ex2}} & \multicolumn{2}{c|}{Example \ref{ex3}} \\ \hline
\nicefrac{1}{2} & 17 & 00:38 & 17 & 00:37 & 17 & 00:37 \\ \hline
\nicefrac{1}{4} & 65 & 02:28 & 65 & 02:33 & 65 & 02:24 \\ \hline
\nicefrac{1}{6} & 145 & 05:32 & 145 & 05:52 & 145 & 05:27 \\ \hline
\end{tabular}
\end{table}

\subsection{An example a geometrical problem of a quasi-convex objective functional}

\begin{example}\label{ex_q}
Suppose we are given several points $A_k$ (the centers of the balls $\omega_k$). It is necessary to find the ball of the smallest radius $R$ that covers these points. In other words, it is necessary to find the center of such a ball so that the maximum distance from the center to these points is the shortest possible. At the same time, we assume that the point (center) $X$ can lie on some set, which is defined by functional constraint \eqref{eq_gm2}, where the coefficients $\alpha_{11},\alpha_{12},\ldots,\alpha_{mn}$ are represented by the matrix \eqref{eq_matrix}. The distance from $X$ to each of the fixed points $A_k$ is determined as follows:
$$
d(X,A_k)=
\begin{cases}
XA_k+(\rho-1)r_k, & \mbox{if } |XA_k|>r_k\text{ ($r_k$ --- radius $\omega_k$, $\rho>1$)},\\
\rho XA_k, & \mbox{otherwise},
\end{cases}
$$
where $d(X,A_k)=:f(x)$ is a concave function ($M_f=\rho$). Note that $d(X, A_k)$ is non-smooth in points X: $|XA_k|=r_k$. For points of non-smoothness we use some element of Clarke subdifferential as analogue of subgradient. \end{example}

Other input data: $n=1000$, $\rho=2$,  $x^0=\frac{(0.1,\ldots,0.1)}{\|(0.1,\ldots,0.1)\|}$. The coordinates of the points $A_k$ are chosen in such a way that $\|A_k\|\in[1;2]$, the number of points $A_k$ is equal to 1000 and $r_k = 1$ for all $k = \overline{1, 100}$.

The results of Example \ref{ex_q} are presented in Table \ref{tab_ex_q}. As one can observe, Algorithm \ref{alg4} works faster than Algorithm \ref{alg2}, however, the estimate rate with regard to the objective function is the same, but with regard to the constraints can be much worse.

\begin{table}[]
\centering
\caption{Comparison of the results of the algorithms, Example \ref{ex_q}.}
\label{tab_ex_q}
\begin{tabular}{|c|c|c|c|c|}
\hline
\multirow{2}{*}{$\varepsilon$} & Iterations & \begin{tabular}[c]{@{}c@{}}Time,\\ MM:SS\end{tabular} & Iterations & \begin{tabular}[c]{@{}c@{}}Time,\\ MM:SS\end{tabular} \\ \cline{2-5}
 & \multicolumn{2}{c|}{Algorithm \ref{alg2}} & \multicolumn{2}{c|}{Algorithm \ref{alg4}} \\ \hline
\nicefrac{1}{2} & 4848 & 00:34 & 17 & 00:02 \\ \hline
\nicefrac{1}{4} & 10132 & 01:17 & 65 & 00:09 \\ \hline
\nicefrac{1}{6} & 15242 & 02:39 & 145 & 00:18 \\ \hline
\nicefrac{1}{8} & 20437 & 03:10 & 257 & 00:36 \\ \hline
\nicefrac{1}{10} & 25593 & 04:23 & 400 & 01:01 \\ \hline
\nicefrac{1}{12} & 30742 & 05:11 & 577 & 01:46 \\ \hline
\end{tabular}
\end{table}

\section{Conclusion}
Summing up, let's remark the conclusions of the article. There was proposed an analogue of adaptive Mirror Descent (\cite{bib_Adaptive}, Section 3.3) for convex programming problems with another  step-size strategy. Estimates of the rate of its convergence were proved. Optimality in terms of lower bounds was stated. Moreover, it was shown, that proposed methods can be used to minimize quasi-convex objective functionals with different levels of smoothness. Also, using the restart technique an optimal method was proposed to solve optimization problems with strongly convex objective functionals. Some numerical experiments were carried out to solve geometrical problems with convex constraints. Advantages of proposed methods were demonstrated during these experiments. Numerical examples for the minimization of quasi-convex functionals were given. As the result, proposed methods work faster then (\cite{bib_Adaptive}, Section 3.3). However,  functional constraint evaluation, generally, can deteriorate: $g(\bar{x}) < M_g \varepsilon$ instead of $g(\bar{x}) < \varepsilon$ in \cite{bib_Adaptive}.

In addition let us show how the main results of the work can be extended to the Lipschitz quasi-convex constraint.

\begin{lemma}\label{LemmTQ}
Lemma \ref{LemmT} is valid for $v_g(y, x_*)$ in the case of quasi-convex objective and constraint.
\end{lemma}

Let us consider the following modification of Algorithm 1 under the assumption of quasi-convexity of the objective functional and constraint. We can use the technique proposed in \cite{Nesterov_A}. Namely, instead of the (sub)gradient $\nabla f$ we should consider the following set
$$\hat{D} f(x) = \{p ~ |~ \langle p, x-y\rangle \geqslant 0 \ \ \forall y\in X: \ f(y) \leqslant f(x)\}.$$
Generally, this set is non-empty, closed and convex cone. Following \cite{Nesterov_A}, we assume that $\hat{D} f(x) \neq \{0\}$. Hereinafter denote $D f(x)$ as one arbitrary vector from $\hat{D} f(x)$: $$D f(x) \in \hat{D} f(x).$$

\begin{algorithm}
\caption{Modification of MDA for quasi-convex constraint}
\label{M-alg2}

    \textbf{IF}\ \ \ \ $g(x^k)\leqslant 
    M_g\varepsilon$\ \ \  \text{(productive steps)}\\
$x^{k+1}=Mirr_{x^k}(h_k^f D f(x^k))$

    \textbf{ELSE}\ \ \ \ \text{(non-productive steps)}\\
    $x^{k+1}=Mirr_{x^k}(h_k^g D g(x^k))$

\end{algorithm}

Let us choose the step-sizes as follows:
$h_k^f=\frac{C_f}{\|D f(x^k)\|_*},$ $h_k^g=\frac{C_g}{\|D g(x^k)\|_*}.$ Denote $N_I, N_J$ as the number of productive and non-productive steps during the work of the Algorithm respectively.
Similar to \cite{bib_Adaptive} (see the proof of Theorem \ref{th2_new_methods}) the next inequality holds:
$$C_f N_I\min\limits_{k\in I}v_f(x_*,x^k)\leqslant \frac{1}{2}\sum\limits_{k\in I}(h_k^f)^2\|D f(x^k)\|^2_2 - C_g\sum\limits_{k\in J} v_g(x_*,x^k)
+$$
$$+\frac{1}{2}\sum\limits_{k\in J}(h_k^g)^2\|D g(x^k)\|^2_2 + \Theta_0^2.$$

Let $C_g=C_f=\varepsilon, \ N\geqslant \frac{2\Theta_0^2}{\varepsilon^2}$. As $g(x^k)\geqslant M_g\varepsilon,$ $k \in J$ and using Lemma \ref{LemmTQ} for constraint $g(x)$ with Lipschitz constant $M_g$ we get
$$-v_g(x_*,x^k) \leqslant (g(x_*)-g(x^k))/M_g\leqslant -g(x^k)/M_g\leqslant -\varepsilon.$$

\begin{theorem}
Let $f$ be quasi-convex, $g$ be quasi-convex with Lipschitz constant $M_g$. Then for Algorithm 3 after \eqref{stop_rule} steps \eqref{eq1_new_methods} holds.
\end{theorem}

\textit{The authors are very grateful to Y. Nesterov for fruitful discussions.}

\section{Appendix: Proof of Theorem 5.1}

\begin{proof}
Function $d_{p}(x)=d\left( \dfrac{x-x^{p}}{R_{p}} \right)$, defined in Algorithm \ref{acc_algorithm2}, is $1$-strongly convex with respect to the norm $\dfrac{\|.\|}{R_{p}}$ for all
$p\geqslant 0$. It is also easy to prove the following inequality
$$\|x^{p}-x_{*}\|^{2} \leqslant R_{p}^{2}\max\{1,M_g\} \quad \forall p \geqslant 0.$$

If $p=0$ the statement holds due to the choosing of $x^0$ and $R_0$. Suppose that $\|x^{p}-x_{*}\|^{2} \leqslant R_{p}^{2}\max\{1,M_g\}$ for some $p$. Let us prove that $\|x^{p+1}-x_{*}\|^{2} \leqslant R_{p+1}^{2}\max\{1,M_g\}$. As $d_{p}(x_*) \leqslant \Omega^{2}\max\{1,M_g\}$, on the restart number $(p+1)$  after no more than
$$N_{p+1}=\left\lceil\dfrac{2 \Omega^{2}\max\{1,M_g\}}{\hat{\varphi}^{2}(\varepsilon_{p+1})}\right\rceil$$
iterations of Algorithm \ref{alg4}, for $x^{p+1}= \bar{x}^{N_{p+1}}$ the next inequalities hold:
$$f(x^{p+1})-f(x_{*}) \leqslant \varepsilon_{p+1}, \quad g(x^{p+1})\leqslant \varepsilon_{p+1}M_g \quad \text{if} \quad \varepsilon_{p+1}=\dfrac{\mu R_{p+1}^{2}}{2}.$$
According to Lemma \ref{acc_lem3}
$$\|x^{p+1}-x_{*}\|^{2} \leqslant \dfrac{2 \varepsilon_{p+1}}{\mu}\max\{1,M_g\} = R_{p+1}^{2}\max\{1,M_g\}.$$
So, for any $p \geqslant 0$ we have proved that
$$
\|x^{p}-x_{*}\|^{2} \leqslant R_{p}^{2}\max\{1,M_g\} = \dfrac{R_{0}^{2}}{2^p}\max\{1,M_g\},
$$
$$
f(x^p)-f(x_*)\leqslant \dfrac{\mu R_{0}^{2}}{2^{p+1}},
\quad
g(x^p) \leqslant \dfrac{\mu R_{0}^{2}M_g}{2^{p+1}}.
$$
Consequently,
$p=\displaystyle{\widehat{p}=\left\lceil\log_{2}\frac{\mu R_{0}^{2}}{2\varepsilon}\right\rceil}$ output $x^p$ is an $\varepsilon$-solution of the problem (\ref{acc_eq14}) and next inequalities hold:
$$\|x^{p}-x_{*}\|^{2} \leqslant R_{p}^{2}\max\{1,M_g\}= \dfrac{R_{0}^{2}}{2^{p}}\max\{1,M_g\} \leqslant \dfrac{2 \varepsilon}{\mu}\max\{1,M_g\}.$$
Let $K$ be the number of iterations of Algorithm \ref{alg4} during the work of Algorithm \ref{acc_algorithm2}, $N_p$ be the total number of iterations of Algorithm \ref{alg4} on the restart number $p$.
As function $\tau: \mathbb{R}^{+}\rightarrow\mathbb{R}^{+}$
increases and for any $\varepsilon > 0$ there exists $\hat{\varphi}(\varepsilon)>0:$ $\tau(\hat{\varphi}(\varepsilon))=\varepsilon$.
It means that
$$
K=\sum_{p=1}^{\widehat{p}} N_{p}=\sum_{p=1}^{\widehat{p}} \left\lceil \frac{2\Omega^2\max\{1,M_g\}}{\hat{\varphi}^{2}(\varepsilon_{p})} \right\rceil \leq
\widehat{p}+\sum_{p=1}^{\widehat{p}}\frac{2\Omega^2\max\{1,M_g\}}{\hat{\varphi}^{2}(\varepsilon_{p})}.
$$

So, the number of iterations of Algorithm \ref{alg4} during the work of Algorithm \ref{acc_algorithm2} will not exceeds
$$\widehat{p}+\sum_{p=1}^{\widehat{p}}\frac{2\Omega^2\max\{1,M_g\}}{\hat{\varphi}^{2}(\varepsilon_{p})},\;\;\text{where}\;\;\varepsilon_{p}=\frac{\mu R^{2}_{0}}{2^{p+1}}.$$
\end{proof}

\end{document}